\documentclass[12pt]{amsart}
\usepackage{amssymb,latexsym}    
\usepackage{fullpage}

\title{ Fixed point adjunctions for equivariant module spectra}
\author{J.~P.~C.~Greenlees}
\address{Department of Pure Mathematics, The Hicks Building, 
Sheffield S3 7RH. UK.}
\email{j.greenlees@sheffield.ac.uk}

\author{B.~Shipley}
\thanks{The first author is grateful for support under EPSRC grant
  number EP/H040692/1. 
    This material is based upon work by the second author supported by the National Science Foundation under Grant No. DMS-1104396.}
\address{Department of Mathematics, Statistics and Computer Science, University of Illinois at
Chicago, 508 SEO m/c 249,
851 S. Morgan Street,
Chicago, IL, 60607-7045, USA}
\email{bshipley@math.uic.edu}

\date{}

\input{xypic}
\newcommand{\Mdef}[2]{\newcommand{#1}{\relax \ifmmode #2 \else $#2$\fi}}

\newcommand{\finbuilds}{\models}

\newcommand{\builds}{\vdash}
\newcommand{\sm }{\wedge}

\newcommand{\tensor}{\otimes}

\newcommand{\Hom}{\mathrm{Hom}}

\Mdef{\bhom}{\mathbf{\hat{H}om}}

\Mdef{\Mod}{\mathrm{mod}}

\newcommand{\st}{\; | \;}

\newtheorem{thm}{Theorem}[section]
\newtheorem{lemma}[thm]{Lemma}
\newtheorem{prop}[thm]{Proposition}
\newtheorem{cor}[thm]{Corollary}

\theoremstyle{definition}

\newtheorem{defn}[thm]{Definition}

\newtheorem{construction}[thm]{Construction}

\newtheorem{example}[thm]{Example}

\newtheorem{remark}[thm]{Remark}

\newcommand{\qqed}{\qed \\[1ex]}
\renewenvironment{proof}[1][\hspace*{-.8ex}]{\noindent {\bf Proof #1:\;}}{\qqed}

\Mdef{\PH} {\Phi^H}
\Mdef{\PN} {\Phi^N}
\Mdef{\PL} {\Phi^L}
\Mdef{\PT} {\Phi^{\T}}

\Mdef{\ef}{E{\cF}_+}
\Mdef{\etf}{\widetilde{E}{\cF}}
\Mdef{\eg}{E{G}_+}
\Mdef{\etg}{\tilde{E}{G}}

\Mdef{\infl}{\mathrm{inf}}
\Mdef{\inflGbarG}{\mathrm{inf}_{G/N}^G}
\Mdef{\defl}{\mathrm{def}}
\Mdef{\res}{\mathrm{res}}
\Mdef{\ind}{\mathrm{ind}}
\Mdef{\coind}{\mathrm{coind}}

\Mdef{\univ}{\mathcal{U}}

\Mdef{\Fp}{\mathbb{F}_p}
\Mdef{\Zpinfty}{\macZ /p^{\infty}}
\Mdef{\Zpadic}{\macZ_p^{\wedge}}

\newcommand{\bi}{\begin{itemize}}
\newcommand{\be}{\begin{enumerate}}
\newcommand{\bc}{\begin{center}}
\newcommand{\bd}{\begin{description}}
\newcommand{\ei}{\end{itemize}}
\newcommand{\ee}{\end{enumerate}}
\newcommand{\ec}{\end{center}}
\newcommand{\ed}{\end{description}}
\newcommand{\dichotomy}[2]{\left\{ \begin{array}{ll}#1\\#2 \end{array}\right.}

\newcommand{\adjunction}[4]{
\diagram
#1:#2 \rrto<0.7ex> &&
#3  \llto<0.7ex> :#4 
\enddiagram}
\newcommand{\RLadjunction}[4]{
\diagram
#1:#2 \rrto<-0.7ex> &&
#3  \llto<-0.7ex> :#4 
\enddiagram}
\newcommand{\lra}{\longrightarrow}

\newcommand{\Gspectra}{\mbox{$G$-{\bf spectra}}}

\newcommand{\abgp}{\mathbf{AbGp}}

\Mdef{\we}{\mathbf{we}}
\Mdef{\fib}{\mathbf{fib}}
\Mdef{\cof}{\mathbf{cof}}
\Mdef{\BI}{\mathcal{BI}}

\Mdef{\A}{\mathbb{A}}
\Mdef{\B}{\mathbb{B}}
\newcommand{\macC}{\mathbb{C}}
\Mdef{\D}{\mathbb{D}}
\Mdef{\E}{\mathbb{E}}
\Mdef{\T}{\mathbb{T}}
\Mdef{\F}{\mathbb{F}}
\Mdef{\G}{\mathbb{G}}
\Mdef{\I}{\mathbb{I}}
\Mdef{\N}{\mathbb{N}}
\newcommand{\macQ}{\mathbb{Q}}
\newcommand{\macR}{\mathbb{R}}
\Mdef{\bbS}{\mathbb{S}}
\newcommand{\macZ}{\mathbb{Z}}

\Mdef{\bA}{\mathbb{A}}
\Mdef{\bB}{\mathbb{B}}
\Mdef{\bC}{\mathbb{C}}
\Mdef{\bD}{\mathbb{D}}
\Mdef{\bE}{\mathbb{E}}
\Mdef{\bF}{\mathbb{F}}
\Mdef{\bG}{\mathbb{G}}
\Mdef{\bH}{\mathbb{H}}
\Mdef{\bI}{\mathbb{I}}
\Mdef{\bJ}{\mathbb{J}}
\Mdef{\bK}{\mathbb{K}}
\Mdef{\bL}{\mathbb{L}}
\Mdef{\bM}{\mathbb{M}}
\Mdef{\bN}{\mathbb{N}}
\Mdef{\bO}{\mathbb{O}}
\Mdef{\bP}{\mathbb{P}}
\Mdef{\bQ}{\mathbb{Q}}
\Mdef{\bR}{\mathbb{R}}
\Mdef{\bS}{\mathbb{S}}
\Mdef{\bT}{\mathbb{T}}
\Mdef{\bU}{\mathbb{U}}
\Mdef{\bV}{\mathbb{V}}
\Mdef{\bW}{\mathbb{W}}
\Mdef{\bX}{\mathbb{X}}
\Mdef{\bY}{\mathbb{Y}}
\Mdef{\bZ}{\mathbb{Z}}

\Mdef{\cA}{\mathcal{A}}
\Mdef{\cB}{\mathcal{B}}
\Mdef{\cC}{\mathcal{C}}
\Mdef{\mcD}{\mathcal{D}} 
\Mdef{\cE}{\mathcal{E}}
\Mdef{\cF}{\mathcal{F}}
\Mdef{\cG}{\mathcal{G}}
\Mdef{\mcH}{\mathcal{H}} 
\Mdef{\cI}{\mathcal{I}}
\Mdef{\cJ}{\mathcal{J}}
\Mdef{\cK}{\mathcal{K}}
\Mdef{\mcL}{\mathcal{L}}

\Mdef{\cM}{\mathcal{M}}
\Mdef{\cN}{\mathcal{N}}
\Mdef{\cO}{\mathcal{O}}
\Mdef{\cP}{\mathcal{P}}
\Mdef{\cQ}{\mathcal{Q}}
\Mdef{\mcR}{\mathcal{R}}
\Mdef{\cS}{\mathcal{S}}
\Mdef{\cT}{\mathcal{T}}
\Mdef{\cU}{\mathcal{U}}
\Mdef{\cV}{\mathcal{V}}
\Mdef{\cW}{\mathcal{W}}
\Mdef{\cX}{\mathcal{X}}
\Mdef{\cY}{\mathcal{Y}}
\Mdef{\cZ}{\mathcal{Z}}

\Mdef{\At}{\tilde{A}}
\Mdef{\Bt}{\tilde{B}}
\Mdef{\Ct}{\tilde{C}}
\Mdef{\Et}{\tilde{E}}
\Mdef{\Ht}{\tilde{H}}
\Mdef{\Kt}{\tilde{K}}
\Mdef{\Lt}{\tilde{L}}
\Mdef{\Mt}{\tilde{M}}
\Mdef{\Nt}{\tilde{N}}
\Mdef{\Pt}{\tilde{P}}

\Mdef{\tA}{\tilde{A}}
\Mdef{\tB}{\tilde{B}}
\Mdef{\tC}{\tilde{C}}
\Mdef{\tE}{\tilde{E}}
\Mdef{\tH}{\tilde{H}}
\Mdef{\tK}{\tilde{K}}
\Mdef{\tL}{\tilde{L}}
\Mdef{\tM}{\tilde{M}}
\Mdef{\tN}{\tilde{N}}
\Mdef{\tP}{\tilde{P}}

\Mdef{\ft}{\tilde{f}}
\Mdef{\xt}{\tilde{x}}
\Mdef{\yt}{\tilde{y}}

\Mdef{\Ab}{\overline{A}}
\Mdef{\Bb}{\overline{B}}
\Mdef{\Cb}{\overline{C}}
\Mdef{\Db}{\overline{D}}
\Mdef{\Eb}{\overline{E}}
\Mdef{\Fb}{\overline{F}}
\Mdef{\Gb}{\overline{G}}
\Mdef{\Hb}{\overline{H}}
\Mdef{\Ib}{\overline{I}}
\Mdef{\Jb}{\overline{J}}
\Mdef{\Kb}{\overline{K}}
\Mdef{\Lb}{\overline{L}}
\Mdef{\Mb}{\overline{M}}
\Mdef{\Nb}{\overline{N}}
\Mdef{\Ob}{\overline{O}}
\Mdef{\Pb}{\overline{P}}
\Mdef{\Qb}{\overline{Q}}
\Mdef{\Rb}{\overline{R}}
\Mdef{\Sb}{\overline{S}}
\Mdef{\Tb}{\overline{T}}
\Mdef{\Ub}{\overline{U}}
\Mdef{\Vb}{\overline{V}}
\Mdef{\Wb}{\overline{W}}
\Mdef{\Xb}{\overline{X}}
\Mdef{\Yb}{\overline{Y}}
\Mdef{\Zb}{\overline{Z}}

\Mdef{\db}{\overline{d}}
\Mdef{\hb}{\overline{h}}
\Mdef{\qb}{\overline{q}}
\Mdef{\rb}{\overline{r}}
\Mdef{\tb}{\overline{t}}
\Mdef{\ub}{\overline{u}}
\Mdef{\vb}{\overline{v}}

\Mdef{\hc}{\hat{c}}
\Mdef{\he}{\hat{e}}
\Mdef{\hf}{\hat{f}}
\Mdef{\hA}{\hat{A}}
\Mdef{\hH}{\hat{H}}
\Mdef{\hJ}{\hat{J}}
\Mdef{\hM}{\hat{M}}
\Mdef{\hP}{\hat{P}}
\Mdef{\hQ}{\hat{Q}}

\Mdef{\thetab}{\overline{\theta}}
\Mdef{\phib}{\overline{\phi}}

\Mdef{\uA}{\underline{A}}
\Mdef{\uB}{\underline{B}}
\Mdef{\uC}{\underline{C}}
\Mdef{\uD}{\underline{D}}

\Mdef{\bolda}{\mathbf{a}}
\Mdef{\boldb}{\mathbf{b}}
\Mdef{\boldD}{\mathbf{D}}

\Mdef{\fm}{\frak{m}}
\Mdef{\fp}{\frak{p}}

\Mdef{\eps}{\epsilon}

\renewcommand{\Et}{\cE_t}

\newcommand{\M}{\bM}

\newcommand{\cKcellM}{\mbox{$\cK$-cell-$\bM$}}

\newcommand{\modcatG}[1]{\mbox{$#1$-mod-$G$-spectra}}

\newcommand{\modcatGN}[1]{\mbox{$#1$-mod-$G/N$-spectra}}

\newcommand{\RmodG}{\modcatG{R}}
\newcommand{\RcellRmodG}{\mbox{$R$-cell-$R$-mod-$G$-spectra}}
\newcommand{\RNcellRNmodGN}{\mbox{$R^N$-cell-$R^N$-mod-$G/N$-spectra}}

\newcommand{\RNmod}{\modcatGN{R^N}}
\newcommand{\infRNmod}{\modcatG{\infl R^N}}

\renewcommand{\Gspectra}{\mbox{$G$-spectra}}

\newcommand{\GNspectra}{\mbox{$G/N$-spectra}}

\newcommand{\efp}{E\cF_+}

\newcommand{\etnotN}{\widetilde{E}[\not \supseteq N]}

\newcommand{\lr}[1]{\langle #1\rangle}

\newcommand{\dimC}{\mathrm{dim}_{\macC}}

\renewcommand{\DH}{D}

\newcommand{\Gbar}{\overline{G}}
\newcommand{\Hbar}{\overline{H}}

\newcommand{\GI}{\mathcal{GI}}

\newcommand{\piM}{\underline{\pi}^G}
\newcommand{\SO}{{\mathcal{SO}}}
\newcommand{\Rbar}{\overline{R}}
\newcommand{\PNequiv}{\mbox{$\Phi^N$-equiv}}
\newcommand{\GspectraoverN}{\mbox{$G$-spectra/$N$}}

\begin{document}

\begin{abstract}
We consider the Quillen adjunction between fixed points and inflation in the
context of equivariant module spectra over equivariant ring spectra,
and give numerous examples including some based on geometric fixed points and some on the Eilenberg-Moore spectral sequence.   
\end{abstract}

\maketitle

\tableofcontents

\section{Introduction}

\subsection{Motivation}

We can view the present paper in two ways. On the one hand we can view
it as an investigation of the formal properties of a basic change of
groups adjunction in equivariant topology. On the other hand we can
view it as giving a powerful context for proving two well known
general results. These are of course two sides of the same coin.

 The
change of groups adjunction is that between fixed points and
inflation, placed in the more general context of modules over ring
spectra, and will be introduced in the Subsection
\ref{subsec:context} below. 

The two well known results are usually considered in rather different
contexts. The first (Sections \ref{sec:GspectraoverN} and
\ref{sec:RmodoverN})
is the fact that the category of $G$-spectra lying
over a normal subgroup $N$ is equivalent to the category of $G/N$-spectra. The second
(Section \ref{sec:EilenbergMoore}) is the Eilenberg-Moore spectral sequence which states that (under
suitable hypotheses) if $X$ is a free $G$-space, then $H^*(X)$ can be calculated from $H^*(X/G)$ as an
$H^*(BG)$-module. Our versions of these results are both expressed as Quillen
equivalences of model categories, and it is striking that these two things are rather
formal consequences of a single formal statement together with the
Cellularization Principle \cite{GScell}. We give a variety of other interesting
specializations of the general results. 

Particular instances of  our general result are central
ingredients in our work \cite{gfreeq2, tnq3} giving algebraic models for categories of
rational equivariant spectra. 

\subsection{Context}
\label{subsec:context}
When group actions are under consideration there is often a valuable
adjunction between passage to fixed points under a normal subgroup $N$
of the ambient group $G$ and inflation from the quotient group $G/N$.

As a first example,  if we are considering vector spaces we obtain an adjunction
in representation theory
$$  \mbox{$G$-Hom}(\inflGbarG  X,Y) \cong  \mbox{$G/N$-Hom}( X,Y^N),  $$
where $Y$ is a vector space with $G$ action, and $X$ is a vector space
with $G/N$ action. The fixed point functor $(\cdot )^N$ and the
inflation functor $\inflGbarG$ require no further explanation in this
case, but we will use this terminology in other cases where the
functors may be less familiar. The adjunction  extends to graded vector
spaces, and to graded vector spaces with a differential (chain
complexes).  If we replace vector spaces by  topological spaces
the adjunction is an elementary statement in equivariant topology.

In these cases, we may choose model structures so that  the adjunction
is a Quillen adjunction, so there is an induced   adjunction between homotopy 
categories
$$ [\inflGbarG  X,Y]^G \cong  [X,Y^N]^{G/N}.  $$

Moving towards cases of concern to us here, we consider categories 
of  equivariant orthogonal spectra \cite{mm} as our
model.  The Lewis-May fixed point adjunction \cite{LMS(M)} 
$$\RLadjunction
{(\cdot )^N}
{\mbox{$G$-spectra}}
{\mbox{$G/N$-spectra}}
{\inflGbarG}
$$
applies to orthogonal spectra \cite[V.3.4]{mm}, and it is a Quillen adjunction. 

The principal purpose of the present paper is to consider a variation
of this adjunction with categories of module spectra over ring
spectra.

\subsection{The case at hand}

The Lewis-May categorical fixed point functor on orthogonal spectra 
is lax symmetric monoidal (preserves smash products precisely, but not units)
\cite[V.3.8]{mm},  so that  if $R$ is a ring $G$-spectrum,  $R^N$ is a
ring $G/N$-spectrum. 

The purpose of this paper is to  consider the 
 Quillen adjunction
$$
\RLadjunction
{\Psi^N}
{\RmodG}
{\RNmod}
{R\tensor_{\infl R^N}(\cdot)} $$
(here and elsewhere we follow Quillen in putting the left adjoint
arrow on top when describing adjunctions). The existence of this follows from the
general discussion on  the interaction of model categories and
monoidal structures in \cite{ss-mon}.   Indeed, in our case 
inflation is strong symmetric monoidal \cite[V.1.5(v)]{mm}, and where
the left adjoint has this property, the functors on spectra induce a
Quillen pair relating the categories of modules. We have used the
notation $\Psi^N$ for the Lewis-May fixed points to highlight the
change of ambient ring from $R$ to $R^N$. 

It turns out that this is rather a rich context, and we make explicit
a number of  examples where the adjunction gives a very
close relationship between the categories in 
a number of interesting cases. The key to this is the Cellularization
Principle \ref{thm:cellprin}  \cite{GScell}, which gives a language to
describe the deviation from equivalence. The examples  include spectra and modules
concentrated over $N$ (giving the classical geometric fixed point equivalence) in Sections 
\ref{sec:GspectraoverN} and  \ref{sec:RmodoverN}, Eilenberg-Mac Lane
spectra in Section \ref{sec:EilenbergMacLane}, sphere spectra in
Section \ref{sec:sphere}  and cochains on a free $G$-space (giving the Eilenberg-Moore
theorem) in Section \ref{sec:EilenbergMoore}. 

\subsection{Relationship to other results}

We proved these results during our work on algebraic models for
rational torus-equivariant spectra, and special cases appeared in
the preprint \cite{tnq3}. Since they are of wider interest,
we present them separately here and refer to  \cite{gfreeq2} and
\cite{tnq3new} for these applications.




\section{Preliminaries}

\subsection{The cellularization principle}
\label{sec:cell}

A key ingredient in applications is the Cellularization Principle of
\cite{GScell}, which states that a Quillen adjunction induces a Quillen equivalence between
cellularized stable model categories in the sense of \cite{hh}
provided the cells are small and  the unit and counit are equivalences on cells. 
This is analogous to the statement that a natural transformation of cohomology
theories that induces an isomorphism on spheres is an equivalence.

\begin{thm}\label{thm:cellprin} {\bf (The Cellularization Principle.)}
Let $\M$ and $\N$ be right proper, stable, 
cellular model
categories with $F: \M \to \N$ a left Quillen 
functor with right adjoint $U$.  Let $Q$ be a cofibrant replacement functor in $\M$ and
$R$ a fibrant replacement functor in $\N$. 
\begin{enumerate}
\item Let $\cK= \{ A_{\alpha} \}$ be a set of objects in $\M$ with $FQ\cK = \{ FQ A_{\alpha}\}$ the corresponding set in $\N$.  
Then $F$ and $U$
induce a Quillen adjunction 
\[ \adjunction{F}{\cKcellM}{\mbox{$FQ\cK$-cell-$\N$}}{U}\]
between the $\cK$-cellularization
of $\M$ and the $FQ\cK$-cellularization of $\N$.

\item If $\cK= \{ A_{\alpha} \}$ is a stable set of small objects in
  $\M$ such that for each $A$ in $\cK$ the object $FQA$ is small in $\N$ and 
the derived unit $QA \to URFQA$ is a weak equivalence in $\M$, 
then $F$ and $U$ induce a Quillen equivalence between the cellularizations:
\[ \cKcellM \simeq_{Q} \mbox{$FQ\cK$-cell-$\N$}. \]

\item If  $\mcL = \{B_{\beta}\}$ is a stable set of small objects in
  $\N$ such that for each $B$ in $\mcL$ the object $URB$ is small in $\M$ and  
the derived counit  $FQURB \to RB$ is a weak equivalence in $\N$,  then $F$ and $U$
induce a Quillen equivalence between the cellularizations:
\[ \mbox{$UR\mcL$-cell-$\M$} \simeq_{Q} \mbox{$\mcL$-cell-$\N$}.\]
\end{enumerate} 
\end{thm}


\subsection{Universal $G$-spaces}
If $\cK$ is a family of subgroups (i.e., a set of subgroups closed
under conjugation and passage to smaller subgroups), there is a 
universal space $E\cK$, characterized up to equivalence by the
fact that $(E\cK)^H$ is empty if $H\not \in \cK$ and is contractible 
if $H\in \cK$. We write $\widetilde{E}\cK$ for the unreduced suspension 
$S^0*E\cK$, so that there is a cofibre sequence
$$E\cK_+\lra S^0 \lra \widetilde{E}\cK . $$


For  normal subgroups $N$ we define certain spaces $E\lr{N}$ 
by the cofibre sequence
$$E[\subset N]_+\lra E[\subseteq N]_+\lra E\lr{N}.$$
We note that in fact $E\lr{N}\simeq EG/N_+\sm \widetilde{E}[\not \supseteq N]$.  

\subsection{Geometric isotropy. }
Recall that the geometric fixed point functor extends the fixed point
space functor in the sense that for based $G$-spaces $Y$
there is a $G/N$-equivalence  $\Phi^N\Sigma^{\infty}Y\simeq
\Sigma^{\infty}(Y^N)$. The isotropy groups of a based $G$-space are the
subgroups so that the fixed points are non-trivial, so it is natural
to consider a stable, homotopy invariant version: the {\em geometric
  isotropy} of a $G$-spectrum $X$ is defined by 
$$\GI (X):=\{ K \st \Phi^KX \not \simeq *\}. $$

Certain $G$-spaces are useful in picking out isotropy
information.  
The geometric isotropy of $\efp$ is $\cF$ and the geometric isotropy 
of $\etf$ is $\cF^c$,  the complement of $\cF$.


\subsection{Geometric fixed points.}
\label{subsec:geomfps}

We say that $X$ {\em lies over $N$} if every subgroup in $\GI (X)$
contains $N$. It is obvious for spaces $X$ that the inclusion $X^N
\lra X $ induces an equivalence $X^N\sm \etnotN  \simeq X \sm \etnotN $
and with a little care there is a generalization to the case when
$X$ is a spectrum.

The relationship between categorical fixed points and  geometric 
points is described in  \cite[II.9]{LMS(M)}.  It follows from the
geometric fixed point Whitehead Theorem that $X$ lies over $N$ if and only if 
$$X\simeq X \sm \etnotN ,$$
and that if $Y$ lies over $N$ then 
$$[X,Y]^G\cong [\Phi^N X, \Phi^NY]^{G/N} .$$

\section{Models of spectra over $N$}
\label{sec:GspectraoverN}
Some of the results in this section are well-known, and we record them
here partly to prepare the way for their module counterparts in
Section \ref{sec:RmodoverN}.

\subsection{Models via localization}
We recall two models for the homotopy category of spectra over
$N$ and their equivalence.

 Choose a $G$-space $\etnotN$, and without change in notation we take
a bifibrant replacement of its suspension spectrum. We
consider the class $\PNequiv$ of morphisms $f:X\lra Y$ so that $
\etnotN \sm f$ is an equivalence; these are called {\em equivalences
over $N$}. 

\begin{construction} \cite[IV.6]{mm}
We form the {\em local model} as the (left) Bousfield localization 
$$\GspectraoverN :=L_{\PNequiv}\mbox{$G$-spectra}. $$
The weak equivalences are the equivalences over $N$ and the
cofibrations are those in the underlying category of $G$-spectra. 
\end{construction}

The inflation functor and its right adjoint, the Lewis-May fixed point
functor are familiar in this context. 

\begin{prop} \cite[VI.5.3]{mm}
Composing the Quillen adjunction 
$$\RLadjunction{(\cdot )^N}{\mbox{$G$-spectra}}
{\mbox{$G/N$-spectra}}{\infl}$$
with localization, we obtain a Quillen equivalence
$$\mbox{$\GspectraoverN$}\simeq \mbox{$G/N$-spectra} .\qqed$$
\end{prop}

\subsection{Models via modules}
\label{subsec:modulemodel}

There is a tempting approach to the category of spectra over $N$ by
viewing them as modules over a ring. However, there is some need for
caution. 

First note that $\etnotN$ is a ring spectrum up to homotopy. 
A $G$-spectrum is an $\etnotN$-module up to homotopy if and only if 
it lies over $N$. Similarly, any map of $G$-spectra over $N$ is compatible up to
homotopy with the homotopy multiplication of $\etnotN$.  Thus the
category of $\etnotN$-modules looks like another model for  the category of $G$-spectra
over $N$. The homotopy category of this category of modules is not
obviously triangulated. 

However we can tighten up the structure. To begin, construct $\etnotN$ as a localization of the
ring spectrum $S^0$. Accordingly, it admits the structure of an
associative ring $G$-spectrum \cite{ekmm} and we may consider  its
category of modules. 
%
%
%
Denote restriction of scalars along the
map of ring spectra $l: S^0\lra \etnotN$ by $l^*$.  Its right adjoint is the
coextension of scalars functor 
$$l_!(M)=F(\etnotN,  M).$$
 Together these give  a Quillen adjunction 
$$\adjunction{l^*}{\mbox{$\etnotN$-mod-$G$-spectra}}{\mbox{$S^0$-mod-$G$-spectra}}{l_!}$$
of module categories. This provides another model for $G$-spectra over
$N$.

\begin{prop}
The restriction and coextension of scalars functors induce a Quillen equivalence
$$\mbox{$\etnotN$-mod-$G$-spectra}\simeq \mbox{$G$-spectra}/N.$$
\end{prop}

\begin{proof}
Since the left Quillen functors are in the right order to be composed, we need only check that for
any cofibrant $\etnotN$-module $C$ and any  $G$-spectrum  $L$, fibrant
in $G$-spectra over $N$, that a map $C\lra l_!L$ is an equivalence in
$\etnotN$-modules if and only if $C\lra L$ is an equivalence over
$N$. Decoding this, we need to show that, in the ambient category of
$G$-spectra, $C\lra F(\etnotN, L)$ is an equivalence if and only if 
$C\sm \etnotN \lra L\sm \etnotN $
is an equivalence. 

Consider the diagram
$$\diagram 
C\sm \etnotN \rto &L\sm \etnotN\\
&F(S^0, L)\uto\\
C\uuto \rto & F(\etnotN, L) \uto
\enddiagram$$
Since $C$ is a cofibrant $\etnotN$-module, the left hand vertical is an
equivalence. Since $L$ is fibrant, \cite[IV.6.13]{mm} shows the right
hand verticals are both equivalences. This gives the desired statement
that  the top horizontal
is an equivalence if and only if the bottom horizontal is an
equivalence. 
\end{proof}

The disappointment is that although  $\etnotN$ is a commutative ring up to
homotopy,  no model for it is  a strictly commutative ring
for reasons described by \cite{McClure}   (or because it is
incompatible with the existence of multiplicative norm maps). These
phenomena are studied systematically in  \cite{Hill.Hopkins}.

\section{Fixed point equivalences}

In this section we explain how to apply the Cellularization Principle \ref{thm:cellprin} to obtain 
interesting  equivalences from the fixed point adjunction
$$
\RLadjunction
{\Psi^N}
{\RmodG}
{\RNmod}
{R\tensor_{\infl R^N}(\cdot)} $$
 on module categories. 
  
\subsection{Towards Quillen equivalences.}\label{subsec:q.eq}
In somewhat simplified notation, consider  the unit and 
counit of the derived adjunction
$$\eta: Y\lra (R\tensor_{\infl R^N}\infl Y)^N 
\mbox{ and }
\eps: R\tensor_{\infl R^N}(\infl X^N)\lra X. $$
We see  that $\eta$ is an equivalence for $Y=R^N$ and that
$\eps$ is an equivalence for $X=R$, so that by the Cellularization Principle \ref{thm:cellprin} 
we always have an
equivalence
$$\RcellRmodG \simeq \RNcellRNmodGN .$$
If $G$ is trivial, $R$ generates the category of $R$-modules, but if 
$G$ is not trivial, $R$ is not usually a generator of the category of equivariant
$R$-modules since we also need the modules of the form $R\sm G/H_+$
for proper subgroups $H$. Similar remarks apply to $R^N$-modules when
$G/N$ is non-trivial.

\subsection{Thick category arguments}
Having established the interest in when categories of equivariant
modules are generated by the ring, we need some more terminology. 

If $R$ {\em does} build every $R$-module, we refer to the module category as {\em
 monogenic}. There are surprisingly many examples of monogenic
equivariant module categories, and in practice the process of building
the modules $R\sm G/H_+$ is of a simple form. 
\begin{defn}
(i) Two $G$-spectra $X$ and $Y$ are {\em $R$-equivalent} ($X\sim_R Y$)
 if there is an equivalence $R\sm X \simeq R\sm Y$ of $R$-modules.

(ii) A $G$-spectrum $X$ is an {\em $R$-retract} of $Y$ if $R\sm X $ is
a retract of $R\sm Y$ as $R$-modules.

(iii) A collection $\cC$ of $G$-spectra is {\em closed under $R$-triangles}
if whenever there is a cofibre sequence $R\sm X \lra R\sm Y\lra R\sm
Z$
of $R$-modules then if two of $X, Y$ and $Z$ lie in $\cC$, so does the
third. 

(iii) A class $\cC$ of $G$-spectra is {\em $R$-thick} if it is closed 
under $R$-equivalence, $R$-retracts and completing $R$-triangles. We
say that $X$ {\em finitely $R$-builds} $Y$ ($X\finbuilds_R Y$) if $Y$ is in
the $R$-thick subcategory generated by $X$.

(iv) A class $\cC$ of $G$-spectra is  {\em $R$-localizing} if it is 
$R$-thick and closed under arbitrary coproducts.  We
say that $X$ {\em $R$-builds} $Y$ ($X\builds_R Y$) if $Y$ is in
the $R$-localizing subcategory generated by $X$.

(v) We say that the category of $R$-modules is {\em strongly
  monogenic} if the $R$-localizing subcategory generated by $S^0$ is
the entire  category of $G$-spectra.
\end{defn}

We note that to show the category of $R$-modules is strongly monogenic,
we need only show that $G/H_+$ is $R$-built by $S^0$ for all subgroups
 $H$. The following observation explains the purpose of the definitions.

\begin{lemma}  If the category of $R$-modules is strongly monogenic, then $R$-cellularization
has no effect on  $\RmodG$, and we have a Quillen equivalence
$$\RcellRmodG \simeq \RmodG.\qqed $$
\end{lemma}

\begin{proof} 
By~\cite[2.2.1]{ss2}, the category of $R$-modules is strongly monogenic if and only if $R$ detects weak equivalences.  Thus, here the $R$-cellular equivalences and the underlying equivalences agree.
\end{proof}

Combining this with the fixed point $R$-module Quillen adjunction from Section~\ref{subsec:q.eq},
we obtain a statement we will use repeatedly. 

\begin{cor}
\label{cor:strongmonsuff}
If $R$ is strongly monogenic we have
$$\RmodG \simeq \RNcellRNmodGN ,  $$
and if in addition $R^N$ is strongly monogenic (for example if $N=G$), then
$$\RmodG \simeq \RNmod . \qqed $$
\end{cor}

\subsection{Strongly monogenic examples}
\label{subsec:strongmonogenic}
Perhaps the first important example of $R$-equivalence is given by
Thom isomorphisms: if $R$ is complex orientable then it is complex
stable, that is, for any complex
representation $V$, we have $R\sm S^V\simeq R\sm S^{|V|}$. So, in particular,
$S^V\sim_R S^{|V|}$, where $|V|$ denotes the underlying vector space
of $V$ with trivial $G$-action. 

\begin{lemma}
\label{lem:RthickT}
If $G$ is a torus and $R$ is complex stable then the category of
$R$-modules is strongly monogenic. 
\end{lemma}

\begin{proof}
It suffices to show $S^0$ $R$-builds  all spectra $G/H_+$.

The proof is built from one cofibre sequence for the circle group and the Thom
isomorphism for $R$. Suppose first that $\alpha : G\lra U(1) $ is a
non-trivial representation, and write $\Gbar = G/K$ where $K=\ker (\alpha )$ for
the relevant circle quotient, noting that there is an equivariant 
homeomorphism $S(\alpha )=\Gbar$. The 
important cofibre sequence is  
$$\Gbar_+=S(\alpha)_+\lra S^0 \lra S^{\alpha}. $$
The Thom isomorphism shows $S^{\alpha}\sim_R S^2$, and the cofibre
sequence shows $\Gbar_+$ is in the $R$-thick subcategory generated
by $S^0$. 

The subgroups occurring as kernels $K$ as in the previous 
paragraph are precisely those of the form 
$K=H\times C$ where $H$ is a rank $r-1$ torus and $C$ is 
a finite cyclic group. Finally we note that for an arbitrary 
subgroup $L$ we have 
$$G/L=G/K_1 \times G/K_2\times \cdots \times G/K_r$$
for suitable $K_i$ occurring as kernels. 

Using the above argument up to $r$ times we see $G/L _+$ is in the
$R$-thick subcategory generated by $S^0$.
\end{proof}

When $G$ is not a torus we need additional restrictions to permit a
similar argument. The flavour is that the group acts trivially on the
coefficients, and there are many examples starting with mod $p$ Borel 
cohomology when $G$ is a $p$-group. 

\begin{lemma}
\label{lem:twogroups}
If $G$ is a $2$-group and for all subgroups $H$ of $G$, $R^H_*(\cdot )$ has Thom isomorphisms for all
real representations, then the category of $R$-modules is strongly monogenic. 
\end{lemma}

\begin{proof}
The essential ingredients in the proof are a cofibre sequence for the group of order $2$ and the Thom
isomorphism for $R$. 

We argue by induction on the order of $G$. The statement is immediate
for the group of order 1, and we may suppose by induction that the
statement has been proved for all proper subgroups of $G$.

Now consider $G$ itself. We must show that for all subgroups $L$ of
$G$, the $G$-space  $G/L_+$ is $R$-built from $R$. This is immediate
if $L=G$, and in other cases we may choose a maximal subgroup $K$ of $G$ so that 
$L\subseteq K \subseteq G$. First we observe that by induction $G/L_+$
is $R$-built from $G/K_+$.  Indeed, 
$G/L_+\simeq G_+\sm_KK/L_+$; by induction $K/L_+$ is $K$-equivariantly
$R$-built from $S^0$, and inducing up to $G$ we see  $G/L_+$ is $G$-equivariantly $R$-built
from $G/K_+$. It remains to show that $G/K_+$ is
$G$-equivariantly built from $S^0$.

Let  $\xi : G\lra O(1)$ be the representation with kernel $K$. Write $\Gbar = G/K$  for the
quotient of order $2$ and note that there is an equivariant homeomorphism
$S(\xi)=\Gbar$. 
The important cofibre sequence is  
$$\Gbar_+= S(\xi)_+\lra S^0 \lra S^{\xi},  $$
The Thom isomorphism shows $S^{\xi}\sim_R S^1$, and the first cofibre
sequence shows $\Gbar_+$ is in the $R$-thick subcategory generated
by $S^0$.
This shows $\Gbar_+$ is in the $R$-thick category generated by $S^0$
as required.

 \end{proof}

\begin{lemma}
\label{lem:pgroups}
If $G$ is a $p$-group, and for all subgroups $H$ of $G$,  $R_H^*(\cdot)$ is complex stable, and $N_G(H)$ acts
trivially on $R_H^*$  then the category of $R$-modules is strongly monogenic. 
\end{lemma}

\begin{proof}
The structure of the proof is precisely the same as for Lemma
\ref{lem:twogroups}, and precisely as in that case, it suffices to
show that $G/K_+$ is $R$-built from $S^0$ for a maximal subgroup $K$.

The difference is that we now need  two cofibre sequences rather than
just one.  Suppose  that $\alpha : G\lra U(1)$ is a  representation
with kernel $K$. Write $\Gbar = G/K$  for the
quotient of order $p$. The first important cofibre sequence is  
$$S(\alpha)_+\lra S^0 \lra S^{\alpha},  $$
and the second is the stable cofibre sequence
$$\Gbar_+ \stackrel{1-g}\lra \Gbar_+ \lra S(\alpha)_+,  $$
where $g$ is a suitable generator of $\Gbar$.
The Thom isomorphism shows $S^{\alpha }\sim_R S^2$, and the first cofibre
sequence shows $S(\alpha)_+$ is in the $R$-thick subcategory generated
by $S^0$. Now since $\Gbar$ acts trivially on $R^K_*$, we obtain a
short exact sequence
$$0\lra  R^K_*\lra R^G_*(S(\alpha)_+)\lra \Sigma R^K_*\lra 0. $$
Since the quotient is free over $R^K_*$, there is a splitting map
$\Sigma R\sm G/K_+ \lra R \sm S(\alpha)_+$ of $R$-modules, and
$$R\sm S(\alpha)_+ \simeq R\sm \Gbar_+\vee \Sigma R\sm \Gbar_+$$
This shows $\Gbar_+$ is in the $R$-thick category as required. 

\end{proof}
\begin{remark}
The action condition can be weakened somewhat, first since it is only
necessary for elements of order $p$ in $W_G(K)$ to act trivially, and
secondly we need only require $1-g$ to act nilpotently. 
\end{remark}

\section{Eilenberg-Mac Lane spectra}
\label{sec:EilenbergMacLane}
In this section we consider the special case where the ambient ring spectrum represents ordinary cohomology, for
which we need to recall some terminology. It is worth bearing in mind
that for finite groups and rational coefficients this is the
general case; see \cite[5.1.2]{ss2} and ~\cite[A.1]{GMTate}.

\subsection{Mackey functors}
\label{sec:Mackey}
We recall that a $G$-Mackey functor is a contravariant additive functor 
$M: \SO_G \lra \abgp$ on the stable orbit category $\SO_G$. For brevity, we
write $M(H)$ for the value on $G/H_+$.  When $G$ is finite, 
this is equivalent \cite[V.9.9]{LMS(M)} to the classical definition of a collection of
abelian groups $M(H)$ related by restrictions and transfers satisfying
the Mackey formula \cite{Dress}. 

Evidently any $G$-equivariant cohomology theory $F$ gives rise to a
graded Mackey functor $\underline{F}_G^*$ whose value at $H$ is
$F_H^*$. Equivalently the homotopy groups of $F$ give a Mackey functor
$\piM_*(F)=[\cdot , F]^G_*$ whose value at $H$ is
$\pi^H_*(F)=[G/H_+, F]^G_*$. A cohomology theory whose value on all 
homogeneous spaces is concentrated in degree 0 is called {\em
  ordinary} or {\em Bredon}. For any Mackey functor $M$, up to equivalence  there is a unique
cohomology theory with coefficient functor $M$, and we write $HM$ for
the representing $G$-spectrum. Thus
$$H_G^*(X;M)=HM_G^*(X)=[X,HM]_G^*.$$
The category of $G$-Mackey functors has a  symmetric monoidal product
which can be swiftly defined by 
$$M\Box N=\piM_0(HM\sm HN),$$
but this can be made categorically explicit using the Day coend
construction \cite{Day, Lewis}. 

A monoid in the category of Mackey functors is called a {\em Green
  functor}. If $\macR$ is a Green functor $H\macR$ is a ring spectrum, and
if the Green functor is commutative, the ring spectrum can also be
taken to be commutative \cite{Ullman}. A prime example is the representable 
Mackey functor $\A_T$ where $T$ is a disjoint union of orbits, 
with $\A_T (H)=[G/H_+, T_+]^G_0$. In particular the Burnside
functor itself is the case when $T$ is a point and has $\A
=\piM_0(S^0)$,  so that  $\A (H)$ is the 
Burnside ring $A(H)$ of virtual $H$-sets. 

\subsection{Modules over an Eilenberg-Mac Lane spectrum}
\label{subsec:EM}
We are now ready to discuss  $R=H\macR$, where $\macR$ is a Green functor,
and we take $N=G$ so we can concentrate on the new features.  In this
case the fixed point spectrum is also an Eilenberg-Mac Lane spectrum
$R^G=(H\macR)^G=H(\macR (G))$.  We are considering the adjunction 
$$\RLadjunction{\Psi^G}{\mbox{$H\macR$-mod-$G$-spectra}}{ \mbox{$H\macR
    (G)$-mod-spectra}}{H\macR\tensor_{H\macR(G)}(\cdot)}. $$

We can be a little more explict. First note that for any connective non-equivariant spectrum
$X$ we have $\piM_0(\infl X)=\pi_0(X)\tensor \A$, where $\A $ is the
Burnside functor. Indeed since Burnside rings are free over $\macZ$, 
both sides are (Mackey functor valued, nonequivariant) homology
theories of $X$; the universal property of $\A$ gives a natural
transformation, and it is an isomorphism for $X=S^0$. Next, note that 
when $X$ is connective,  
$$[\infl X , H\macR]^G=H\macR^0_G(\infl X)=\Hom (\pi_0(X), \macR(G)).  $$
First note that both sides depend only on the 1-skeleton of $X$, and
now  it suffices to note that there is a natural comparison map which
is an isomorphism for a wedge of copies of $S^0$, and both sides are left exact. 
It therefore follows that the counit of the adjunction 
$$\infl R^G=\infl H\macR (G)\lra H\macR =R$$
is determined up to homotopy by being the scalar extension of the
identity map in $ \pi^G_0$.

Note that if $G$ is finite $H\macR^K_*(Y_+)$ is concentrated in degree
0 for any subgroup $K$ and any finite $K$-set $Y$. Accordingly,  $H\macR$-module maps between the generators
$G/H_+ \sm H\macR$ of the category of $H\macR$-modules are concentrated in
degree 0 and Morita theory~\cite[5.1.1]{ss2} shows that we have a purely algebraic
description of the categories concerned: 
$$\mbox{$H\macR$-mod-$G$-spectra}\simeq  \mbox{$\macR$-mod-$G$-Mackey},$$  
 where `mod' in the algebraic setting refers to differential graded modules. 
Thus we are considering
$$\RLadjunction{ev_G}{\mbox{$\macR$-mod-$G$-Mackey}}{ \mbox{$\macR
 (G)$-mod-Abgrps}}{\macR\tensor_{\macR(G)}(\cdot)}. $$ 
This is a Green ring analogue of  one of the  adjoints of evaluation
\cite{RemMackey}, and it would be interesting to give a more systematic
account of the case $N\neq G$.

\begin{remark}
An instructive {\em non-example} is to take the ring $R=\macQ[G]$, so that
$R^N=\macQ [G/N]$. This is not an equivariant ring spectrum since, with  the 
implied action of $G$ on $R$ by left translation, the multiplication on
$R$ is {\em not} $G$-equivariant, and the unit map is {\em not} $G$-equivariant.

The alternative to this is to consider the conjugation action on $G$,
and then we obtain a genuine equivariant ring. Of course this applies
equally to the group ring $R=k[G^c]=k\sm G^c_+$ for any commutative 
ring spectrum $k$.
\end{remark}

\section{The sphere spectrum}
\label{sec:sphere}
Perhaps the first naturally occurring example has
$R=S^0$ and $N=G$. We first observe that by Segal-tom Dieck splitting
we have 
$$R^G=(S^0)^G\simeq \bigvee_{(H)}BW_G(H)^{L(H)}, $$
where $L(H)$ is the representation on the tangent space to the
identity coset of $G/H$. This is rather a complicated spectrum, and 
utterly different from the sphere. 

For simplicity we restrict attention to the rational case.

\subsection{Finite groups}
\label{subsec:fingps}
When $G$ is finite, and we work rationally,  the sphere is the
Eilenberg-Mac Lane spectrum for the rationalized Burnside functor $\A$:
$S^0\simeq_{\macQ} H\A$, and this is a special case of the material in Subsection
\ref{subsec:EM}. In particular $R^G$ (which is a product of rational
spheres) is more naturally described as the Eilenberg-Mac Lane spectrum for the
 rationalization of the Burnside ring
 $A(G)=\pi^G_0(S^0)=\pi_0((S^0)^G)$.
Thus we are considering 
the usual fixed point adjunction and the adjunction between the Quillen equivalent algebraic categories from Section \ref{subsec:EM}
$$\diagram
\mbox{$S^0_{\macQ}$-module-$G$-spectra} \dto^{\simeq}
\rrto<-0.7ex> &&
\mbox{$(S^0_{\macQ})^G$-module-spectra} \dto^{\simeq} \llto<-0.7ex> \\
\mbox{$\A_{\macQ}$-module-$G$-Mackey} 
\rrto<-0.7ex> &&
\mbox{$A(G)_{\macQ}$-modules.}  \llto<-0.7ex> 
\enddiagram $$
It is simplest to discuss the lower purely algebraic adjunction, but
of course the discussion has a complete topological counterpart. 

The mark homomorphism $A(G) \lra \prod_{(H)} \macZ$ is a rational
isomorphism, so we have a complete set of idempotents $e_H\in
A(G)_{\macQ}$, and we may split both sides. On the right, by definition
of the idempotents,  $e_HA(G)_{\macQ}\cong \macQ$. On the left, 
evaluation at $G/H$ gives an isomorphism \cite[Theorem A.9]{GMTate}
$$\mbox{$e_H\A_{\macQ}$-modules}\stackrel{\cong}\lra  \mbox{$\macQ
  W_G(H)$-modules}. $$
Picking one subgroup $H$ and considering just the $H$th factor by applying the idempotent $e_H$ to the four categories in the above diagram gives the top two lines in the next diagram.  For the bottom line we then use the computations from this paragraph to simplify the algebraic categories.
$$\diagram
\mbox{$e_HS^0_{\macQ}$-module-$G$-spectra} \dto^{\simeq}
\rrto<-0.7ex> &&
\mbox{$e_H(S^0_{\macQ})^G$-module-spectra} \dto^{\simeq} \llto<-0.7ex> \\
\mbox{$e_H\A_{\macQ}$-module-$G$-Mackey} \dto^{\simeq}
\rrto<-0.7ex> &&
\mbox{$e_HA(G)_{\macQ}$-modules}  \dto^{\simeq}\llto<-0.7ex> \\
\mbox{$\macQ W_G(H)$-modules} 
\rrto<-0.7ex> &&
\mbox{$\macQ$-modules.}  \llto<-0.7ex> 
\enddiagram $$
The right adjoint at the top is $G$-fixed points, in the middle it is
evaluation at $G/G$ and at the bottom it $W_G(H)$-fixed points. 
%

If $W_G(H)$ is non-trivial, there is more than one simple $\macQ
W_G(H)$-module, so the $H$th factors on the left and right of the
bottom row are inequivalent for most subgroups $H$. 
Indeed, in the bottom row  the right adjoint takes
$W_G(H)$-fixed points, so that on the right we have just 
the part corresponding to the trivial module.  In model theoretic
terms, the Cellularization Principle shows that the category $\macQ
W_G(H)$-modules with trivial action (equivalent to the category of
$\macQ$-modules) is
the cellularization of the category of all $\macQ W_G(H)$-modules with
respect to $\macQ$.

Reassembling the pieces by considering the product categories,
we reach the conclusion that for $R=S^0_{\macQ}$, the
fixed-point-inflation adjunction is 
$$\RLadjunction{R}{\prod_{(H)} \mbox{$\macQ W_G(H)$-mod}}
{\prod_{(H)} \mbox{$\macQ$-mod}}{L}, $$
which is only an equivalence when $G$ is trivial.

\subsection{The circle group}
When $G$ is the circle group, the ring  $\pi_*^G(S^0)$ has trivial
multiplication, so most information is captured in non-zero homological
degree. In any case, 
the fixed-point-inflation adjunction is not a Quillen equivalence
either (by connectivity the counit is not an equivalence for a free cell $G_+$). 
To get a Quillen equivalence with a category of modules over a
non-equivariant ring spectrum it is more effective to use a Koszul
dual approach as in Example \ref{eg:DEG} below, and as given in detail
in \cite{s1q, tnq3}. 

\section{Models of categories of modules over $N$}
\label{sec:RmodoverN}
A very satisfying class of examples is a generalization of the basic
property of geometric fixed point spectra. This takes the
results in Section \ref{sec:GspectraoverN} into the context of
modules over a ring spectrum.

To start we consider an arbitrary ring $G$-spectrum $R$ and introduce the notation
\begin{itemize}
\item $\macC_0=\mbox{$R$-module-$G$-spectra}$
\end{itemize} 
Now we may consider a model of $R$-module $G$-spectra
over $N$, just as for the sphere in Section \ref{sec:GspectraoverN}.
\begin{itemize}
\item $\macC_1=L_{\Phi^N}(\mbox{$R$-module-$G$-spectra})$
\end{itemize}

Now the $G/N$-spectrum $\Rbar=\Phi^NR$ is the $N$-fixed points of the 
associative ring spectrum $R'=R\sm \etnotN$, and hence it is an associative
ring with a module category

\begin{itemize}
\item $\macC_2$=$\Phi^NR$-module-$G/N$-spectra
\end{itemize}
If $\Rbar$ admits the structure of a commutative ring, then the
module category admits a symmetric monoidal 
tensor product. 
Note also that there is a map $R\lra R'$ of
associative ring spectra. If $R$ is concentrated over $N$ then this is
a weak equivalence, and the module categories are Quillen equivalent.

If we {\em assume} that $R$ is a ring $G$-spectrum over $N$,
 then we obtain a ring $G/N$-spectrum $R^N\simeq
(R')^N\simeq \Phi^NR$ which is strictly commutative if $R$ is, and 
 $R\simeq  \infl \Phi^N R \sm \etnotN$.  One way to reach this setting
 is to start with a ring $G/N$-spectrum $\Rbar$ and then {\em define} $R=\infl \Rbar \sm \etnotN$, noting
that $R\simeq R'$ in this case  (and that $R$ need not be strictly commutative).

\begin{thm}
\label{thm:NfixedmodoverN} 
If $R$ is a ring $G$-spectrum concentrated over $N$ then the fixed
point adjunction gives  a Quillen equivalence 
$$\macC_0=\RmodG \simeq \RNmod\simeq \macC_2 $$
and these are also Quillen equivalent to $\macC_1=L_{\Phi^N}(\mbox{$R$-module-$G$-spectra})$. 
\end{thm}

\begin{remark}
\label{rem:spectraoverN}
(i) If $X$ is an $R$-module, then $X$ is a retract of 
$R \sm X$, so that if $R$ lies over $N$ so too do its modules. 

(ii) If $X$ lies over $N$ then $X^N\simeq \Phi^NX$. Accordingly 
$R^N\simeq \Phi^NR$ and Lewis-May fixed points may be replaced by 
geometric fixed points throughout. 
\end{remark}

\begin{proof}
By the discussion in Subsection \ref{subsec:q.eq} we need only check
that  the cells generating the two module categories correspond. 

Indicating the image of
a subgroup in $G/N$ by a bar, we need only remark that the 
extensions of the cells $R^N \sm \Gbar/\Hbar_+$ which generate the 
category of $R^N$-module $\Gbar$-spectra generate $R$-module
$G$-spectra. This is because the cells $G/H_+\sm R$ with $H$ not 
containing $N$ are contractible. 

For the final statement we note that since $R$ is
concentrated over $N$, $R\simeq R \sm \etnotN$ so that weak 
equivalences coincide with weak equivalences over $N$. This means that
the $\etnotN$-Bousfield localization of $R$-module-$G$-spectra  is a Quillen
equivalence. 
\end{proof}

\section{Eilenberg-Moore examples}
\label{sec:EilenbergMoore}
An example of rather a different character is given by taking a
$G/N$-equivariant ring spectrum $\Rbar$ and an $N$-free $G$-space $E$
and then taking $R=F(E_+, \infl \Rbar)$, so that $R^N=F(E/N_+,\Rbar)$. 
We shall concentrate on $N=G$ with $E=EG$. 
Our focus is on the case  $\Rbar = Hk$ for a 
commutative ring $k$, but we also comment on $\Rbar =S^0$.

Turning to the first example of the above type we take $N=G$ and  
$$R=F(EG_+, \infl Hk)\simeq F(EG_+, H\underline{k}) \mbox{ so that }
  R^G=F(BG_+,Hk) $$
where $\underline{k}$ is the Mackey functor constant at $k$. In
general we write $C^*(X;k) =F(X,Hk)$ since this is a spectrum whose
homotopy is $H^*(X;k)$. Thus $R=C^*(EG;k)$ and $R^G=C^*(BG;k)$; we
will usually not mention the coefficients explicitly. 

Omitting the coefficients $k$ from the notation, we consider the
adjunction
$$\RLadjunction{\Psi^G}{\mbox{$C^*(EG)$-mod-$G$-spectra}}{
  \mbox{$C^*(BG)$-mod-spectra}}{C^*(EG)\tensor_{C^*(BG)}(\cdot)}. $$
Because of existing machinery,  we won't use the methods of Subsection 
\ref{subsec:strongmonogenic}, but instead discuss directly for which
$G$-spaces $X$ the unit of the adjunction is  an equivalence for the 
$C^*(EG)$-module $C^*(EG \times X)$. Since $C^*(EG\times X) \simeq
C^*(EG)\sm D(X_+)$ when $X$ is a finite complex,  discussion of modules of 
this
form will in particular include the generators $C^*(EG)\sm G/H_+$.
Note that if $Y$ is  a free $G$-space (such as $EG\times X$), we have
 $C^*(Y)^G\simeq C^*(Y/G)$ so that the counit here,
$$ C^*(EG)\tensor_{C^*(BG)} C^*(Y/G) \to C^*(Y),$$
is an embodiment of the Eilenberg-Moore 
spectral sequence for the fibration
$$ X \lra EG\times_G X \lra EG\times_G *=BG. $$
When the Eilenberg-Moore spectral sequence converges, 
$k \tensor_{C^*(BG)} C^*(Y/G)\simeq C^*(Y)$, so the counit is an
equivalence. 

\begin{cor}
If $G$ is connected or if $\pi_0G$ is a $p$-group and  $p^N=0$ on $k$
then we have a Quillen-equivalence 
$$\mbox{$C^*(EG)$-mod-$G$-spectra}\simeq \mbox{$C^*(BG)$-mod-spectra}.$$
\end{cor}

\begin{proof}
Under the given hypotheses, the Eilenberg-Moore spectral sequence is
convergent, so that  $k\tensor_{C^*(BG)}C^*(Y/G) \simeq C^*(Y)$ for any free
$G$-space $Y$ \cite{Dwyer, Mandell}. Alternatively, since we have Thom
isomorphisms, we may argue by Lemmas \ref{lem:twogroups} and
\ref{lem:pgroups} that the category of $R$-modules is strongly
monogenic. 

In either case, the counit is an equivalence for a set of small
generators, so  it follows from the Cellularization Principle
\ref{thm:cellprin} that the  adjunction is a Quillen equivalence.  
\end{proof}

To highlight the significance of this result we give two examples
where we do not obtain an equivalence. 

\begin{example} If $G$ is finite and we take rational
coefficients, then the cofibration $EG_+\lra S^0 \lra \tilde{E}G$
splits and we have a rational equivalence  $C^*(EG)\simeq
EG_+$. Accordingly, $C^*(EG)$  is free and  
$$\mbox{$C^*(EG)$-mod-$G$-spectra}\simeq \mbox{$G$-free-$S^0_{\macQ}$-mod-$G$-spectra}
\simeq \mbox{$\macQ G$-mod}. $$
On the other hand, 
$$\mbox{$S^0_{\macQ}$-mod-spectra}\simeq \mbox{$\macQ $-mod}. $$
Note that the same conclusion was reached in Subsection
\ref{subsec:fingps}. In any case,  the category of $\macQ G$-modules is not equivalent to the
category of $\macQ$-modules unless $G$ is trivial. 
\end{example}

\begin{example}
\label{eg:DEG}
Despite the truth of the Segal conjecture, the case with coefficients
in a sphere is not so well behaved. Here we take  
$R=DEG_+=F(EG_+, S^0)$ with  $N=G$ and
$R^G=DBG_+=F(BG_+, S^0)$. The adjunction then takes the form 
$$\RLadjunction{\Psi^G}{\mbox{$DEG_+$-mod-$G$-spectra}}{
  \mbox{$DBG_+$-mod-spectra}}{DEG_+\tensor_{DBG_+}(\cdot)}. $$ 
Again this is sometimes a Quillen equivalence and sometimes not. If we
have rational coefficients, then $S^0\lra H\macQ$ is a non-equivariant
equivalence, so we are back in the Eilenberg-Moore situation, where we know
it is a Quillen equivalence if $G$ is a torus, but not if $G$ is a
non-trivial finite group.  
  
Carlsson's theorem \cite{Carlsson} implies \cite{DEG2,DEG3} that when
$X$ is a based finite complex,
$D(EG_+ \sm X)\simeq (DX)_I^{\wedge}$   
where $I$ is the augmentation ideal of the
Burnside ring. In particular, if $G$ is a $p$ group and we take
coefficients in the $p$-adic sphere we have 
$$R=F(EG_+,
(S^0)_p^{\wedge})\simeq  ((S^0)_p^{\wedge})_I^{\wedge}\simeq
  (S^0)_p^{\wedge}$$
since some power of $I$ lies in the ideal $(p)$.     
Accordingly this puts us back in the situation of Section
\ref{sec:sphere}. 

Note that since rationalization commutes with tensor products, we have
$$(R\tensor_{R^G}M)_{\macQ} \simeq R_{\macQ}\tensor_{(R_{\macQ})^G}M_{\macQ}. $$
From our analysis of the rational case in Section \ref{sec:sphere} we see that we do not obtain 
a Quillen equivalence when $R=F(EG_+, (S^0)_p^{\wedge})$. 
\end{example}





\section{The ring spectrum $\protect DE{\mathcal F}_+$}

In this section we work rationally, and consider the special case when $G$ is a torus. 

One example of importance in the study of rational spectra is the
counterpart of the Eilenberg-Moore example for almost free spaces
(i.e., spaces whose isotropy groups are all finite). 
Thus we let $\cF$
denote the family of finite subgroups of $G$, and we recall the
splitting theorem from \cite{EFQ}: 
$$\efp \simeq \prod_{F} E\lr{F},  $$
where the product is over  finite subgroups $F$.

Taking rational duals, we find
$$R=\DH \efp \simeq \prod_{F} \DH E\lr{F}
\mbox{ and } R^G= \DH (\efp)^G\simeq \prod_{F\in \cF} \DH (B(G/F)_+). $$
Accordingly, this is really a question of assembling the information
from the different finite subgroups.
Since $G$ is abelian we have $E\lr{H}=EG/H_+\sm \widetilde{E}[\not
\supseteq H]$, so the difference between $EG/H_+$ and $E\lr{H}$ is easy to
describe. When $H$ is also finite and coefficients are rational, this
simplifies further.

\begin{lemma}
\label{lem:suspclosure}
With $R=D\efp$, provided $G$ is a torus and we use rational coefficients,  
 the spheres of complex representations are $R$-built
from $S^0$. 
 \end{lemma}

\begin{proof}
We shall show that for any complex representation $V$ the suspension
 $S^V\sm D\efp$ is a finite wedge of retracts of integer suspensions of $D\efp$.

Since $S^V\sm D\efp \simeq D(S^{-V}\sm \efp)$, it suffices to deal 
with the undualized form. 

Now suppose given a complex representation $V$. Since we are working
over the rationals, the classical  Thom isomorphism gives an equivalence
$$S^V \sm E\lr{F} \simeq S^{|V^F|}\sm E\lr{F}$$
for each finite subgroup $F$.

Finally, we divide the finite subgroups into sets
according to $\dimC (V^F)$. Of course there are only finitely many of 
these, and we may apply the corresponding orthogonal idempotents to 
consider the sets separately.
For subgroups $F$ with $\dimC (V^F)=k$,  the suspension is the same in
each case, so that taking products over these subgroups, we find 
$$\prod_FD(E\lr{F} \sm S^V \sm X)\simeq S^{-k} \sm
\prod_F D(E\lr{F} \sm X), $$
as required.
\end{proof}

The argument of Lemma \ref{lem:RthickT} shows that the category of
$D\efp$-modules is strongly monogenic. 

\begin{cor}
\label{cor:suspclosure}
With $R=D\efp$, $G$ a torus and rational coefficients we have an
equivalence
$$\mbox{$D\efp$-mod-$G$-spectra}\simeq 
\mbox{$\prod_F C^*(BG/F)$-mod-spectra}.\qqed$$
\end{cor}


\begin{remark}
We note that we can apply idempotents to take the factor with $F=1$
and recover the Eilenberg-Moore example of Section
\ref{sec:EilenbergMoore} in the case that the ambient group is a torus
and the coefficients are rational.  
\end{remark}

\end{document}